\documentclass[letterpaper,12pt]{amsart}
\title[Disentangling Number]{The Disentangling Number for Phylogenetic Mixtures}

\author{Seth Sullivant}
\address{Department of Mathematics \\
North Carolina State University, Raleigh, NC, 27695}
\email{smsulli2@ncsu.edu}

\date{}

\usepackage{latexsym,array,delarray,amsthm,amssymb,epsfig}%eufrak
\usepackage{color}

\theoremstyle{plain}
\newtheorem{thm}{Theorem}

\newtheorem{prop}[thm]{Proposition}

\theoremstyle{definition}
\newtheorem{defn}[thm]{Definition}

\theoremstyle{remark}
\newtheorem*{rmk}{Remark}

\newcommand{\zz}{\mathbb{Z}}
\newcommand{\nn}{\mathbb{N}}

\newcommand{\rr}{\mathbb{R}}

\newcommand{\calr}{\mathcal{R}}
\newcommand{\calt}{\mathcal{T}}

\newcommand{\inD}[1][\relax]{\def\argone{#1}\def\temprelax{\relax}
  \ifx\argone\temprelax\right.\else\,\middle|#1\right.{}\fi}

\begin{document}

\begin{abstract}
We provide a logarithmic upper bound for the disentangling number on unordered lists of leaf labeled trees.  This results is useful for analyzing phylogenetic mixture models.  The proof depends on interpreting multisets of trees as high dimensional contingency tables.
\end{abstract}

\maketitle

%%%%%%%%%%%%%%%%%%%%%%%%%%%%%%%%%%
%%%%%%%%%%%%%%%%%%%%%%%%%%%%%%%%%%
%%%%%%%%%%%%%%%%%%%%%%%%%%%%%%%%%%
%%%%%%%%%%%%%%%%%%%%%%%%%%%%%%%%%%

For a set $X$ of leaf labels let
 $\calt_{X}$ be the set of trivalent leaf labeled trees  (see \cite{Semple2003} for background on leaf labeled trees).  Typically, the labels come from the set $[n] = \{1,2, \ldots, n\}$.  For $K \subseteq X$ and $T \in \calt_{X}$, let $T|_{K}$ denote the trivalent tree $T|_{K} \in \calt_{K}$ obtained by restricting $T$ to the label set of leaves $K$, contracting vertices of degree two as necessary to obtain trivalent tree.  
Let $\calt_{X,r}$ be the set of unordered lists of length $r$ of elements of  $\calt_{X}$.  Note that these elements need not be distinct.  For $S \in \calt_{X,r}$, with $S = (T_{1}, \ldots, T_{r})$, let $S|_{K} = (T_{1}|_{K}, \ldots, T_{r}|_{K})$.

\begin{figure}[h]
\includegraphics{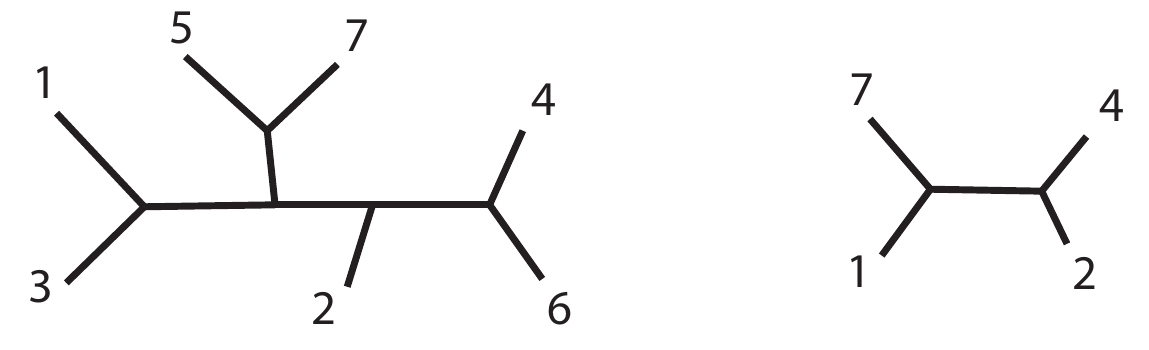} \caption{The right-hand tree is $T|_{\{1,2,4,7\}}$ for the tree $T$ on the left.}
\end{figure}

\begin{defn}
Let $S_{1}, S_{2} \in \calt_{X,r}$ with $S_{1} \neq S_{2}$.  A subset $K \subseteq X$ is said to disentangle $S_{1}$ and $S_{2}$ if $S_{1}|_{K} \neq S_{2}|_{K}$.  Let $d(S_{1}, S_{2})$ be the cardinality of the minimum disentangling set of $S_{1}$ and $S_{2}$.  The disentangling number $D(r)$ is
$$
D(r) = \max_{n \in \nn }   \max_{S_{1} \neq S_{2} \in \calt_{[n],r}} d(S_{1}, S_{2}).
$$
\end{defn}

Humphries \cite{Humphries2008} proved that the disentangling number exists, and gave the bounds
$$
3 ( \lfloor \log_{2} r \rfloor + 1)  \leq D(r)  \leq 3r.
$$
At present, the only exactly known values of the disentangling number are $D(1) = 4$ and $D(2) = 6$ \cite{Matsen2008}.  This first value $D(1) = 4$, is usually stated as saying that ``the quartets determine the tree'' (see e.g.~\cite{Semple2003}).  

The main motivation for studying the disentangling number is that it can be used as a tool in proofs of the identifiability of the tree parameters in phylogenetic mixture models.  Indeed, if it can be shown, for some value $s \geq D(r)$, that the tree parameters of $r$ class mixtures on $s$-leaf phylogenetic trees are (generically) identifiable under some phylogenetic model, then the tree parameters are (generically) identifiable for $r$ class mixtures all trees with $t \geq s$ leaves.  For example, the known value of $D(2) = 6$ was used in the proof of generic identifiability of the tree parameters of $2$-tree Jukes-Cantor and Kimura $2$-parameter mixture models \cite{Allman2010}.

We provide the following improved upper bound on the disentangling number, which is within one of the optimal possible value.

\begin{thm} \label{thm:main}
$D(r) \leq 3 ( \lfloor \log_{2}(r) \rfloor + 1) + 1 $
\end{thm}

To prove Theorem \ref{thm:main} we first reduce to rooted binary trees, as follows.  Let $\calr\calt_{X}$ denote the set of leaf labeled rooted binary trees on leaf label set $X$.  For $T \in \calr\calt_{X}$ and $K \subseteq X$, let $T|_{K}$ be the induced binary rooted tree on leaf label set $K$, with edges contracted as appropriate to obtain a rooted binary tree.  Let $\calr\calt_{X,r}$ be the set of unordered lists of length $r$ of elements of  $\calr\calt_{X}$.  Note that these elements need not be distinct.  For $S \in \calr\calt_{X,r}$, with $S = (T_{1}, \ldots, T_{r})$, let $S|_{K} = (T_{1}|_{K}, \ldots, T_{r}|_{K})$.  Define the rooted disentangling number $RD(r)$ in an analogous way to the disentangling number.

\begin{prop}
The  disentangling and rooted disentangling numbers satisfy:  $D(r) \leq RD(r) + 1.$
\end{prop}

\begin{proof}
Let $n \geq RD(r)$, and consider a set $X$ of cardinality $n +1$, e.g. $X = \{0\} \cup [n]$.   Let $S_{1}, S_{2} \in \calt_{X,r}$.  Choosing the node $0$ (or any other leaf) as a root node, we arrive at sets $\tilde{S}_{1}, \tilde{S}_{2} \in \calr\calt_{[n],r}$.  By definition of the rooted disentangling number, there is a $K \subseteq [n]$ of $RD(r)$ elements such $\tilde{S}_{1}|_{K} \neq \tilde{S}_{2}|_{K}$.   This implies that the set $\{0\} \cup K$ satisfies $S_{1}|_{ \{0\} \cup K}  \neq S_{2}|_{\{0\} \cup K}.$ 
\end{proof}

Theorem \ref{thm:main} then follows as a corollary of the following result.

\begin{thm}\label{thm:submain}
$RD(r)  = 3 ( \lfloor \log_{2}(r) \rfloor + 1).$
\end{thm} 

To prove  the inequality $RD(r)  \leq 3 ( \lfloor \log_{2}(r) \rfloor + 1)$ of Theorem \ref{thm:submain}, we use the known value $RD(1) = 3$ (i.e.~rooted triples determined a rooted tree  (see e.g.~\cite{Semple2003})) to encode multisets of trees as high dimensional contingency tables.  To this end, consider the $3^{\#X \choose 3}$-dimensional space 
$$Q_{X}:= \rr^{3^{\#X \choose 3}} =  \bigotimes_{\{i,j,k\} \in {X \choose 3} }\rr \langle e_{i|jk}, e_{j|ik}, e_{k|ij} \rangle .$$  
Coordinates on this space are indexed by the lists of ${X \choose 3}$-rooted triplets.  Each rooted  $X$-leaf trivalent tree $T$ gives rise to a uniquely determined standard unit vector 
$$e_{T} :=  \otimes_{\{i,j,k\} \in {X \choose 3} } e_{T|_{\{i,j,k\}}} \in  Q_{X}.$$  
This uniqueness is a consequence of the fact that rooted triples in a rooted tree uniquely determine the tree.   

An unordered list of trees $S = (T_{1}, \ldots, T_{r}) \in \calt_{X,r}$ gives rise to a nonnegative integer array $$u_{S} = e_{T_{1}} + e_{T_{2}} + \cdots + e_{T_{r}}$$
 in $Q_{X}$, whose $1$-norm is equal to $r$.  Furthermore, the list $S$ can be recovered from the vector $u_{S}$.  We will use this encoding of sets of trees to prove Theorem \ref{thm:submain}.
 
Let $K$ be a finite set.  For each $k \in K$, let $d_{k} \in \nn_{>1}$ and consider the space
$$
\rr^{d_{K}}  =  \bigotimes_{k \in K} \rr^{d_{k}},
$$ 
with the standard unit vectors $\otimes_{k \in K}  e_{j_{k}} $.  For each $L \subseteq K$, we get a linear map
$$
\pi_{L} : \rr^{d_{K}} \rightarrow \rr^{d_{L}},   \otimes_{k \in K}  e_{j_{k}}  \mapsto \otimes_{k \in L}  e_{j_{k}}.
$$
Given $u \in  \rr^{d_{K}} $, $\pi_{L}(u)$ is called the $L$-marginal of $u$.  If $\Gamma = \{L_{1}, \ldots, L_{s}\}$ is a set of subsets of $K$, we get an induced linear map
$$
\pi_{\Gamma}: \rr^{d_{K}} \rightarrow  \bigoplus_{i =1}^{s} \rr^{d_{L_{i}}},  \quad u \mapsto \pi_{L_{1}}(u) \oplus \cdots \oplus \pi_{L_{s}}(u). 
$$
which is the linear transformation that computes the $\Gamma$-marginals of $u$.  Suppose now that $\Gamma$ is closed downward, that is $L \in \Gamma$ and $L' \subseteq L$ implies that $L' \in \Gamma$.  Such a $\Gamma$ is called a simplicial complex.  The elements of $\Gamma$ are called the faces of $\Gamma$.  

\begin{thm}\cite{Kahle2010} \label{thm:kahle}
Let $\Gamma$ be a simplicial complex, let $s$ be the cardinality of the smallest $S \subset K$ not in $\Gamma$, and $u \in \ker_{\zz} \pi_{\Gamma}$ with $u \neq 0$.  Then $\|u\|_{1} \geq 2^{s}$.
\end{thm}

\begin{proof}[Proof that $RD(r)  \leq 3 ( \lfloor \log_{2}(r) \rfloor + 1)$]

Fix $r$, and suppose that $D(r) > g(r) :=  3 ( \lfloor \log_{2}(r) \rfloor + 1)$.    Then there exists two unordered lists of rooted binary trees $S_{1}, S_{2} \in \calt_{[n], r}$ for $n >  g(r)$ such that for all $k \leq g(r)$ and  $K \in { [n] \choose k}$, we have $S_{1}|_{K} = S_{2}|_{K}$.  

Let $\Gamma_{r} $ be the simplicial complex with ground set ${[n] \choose 3}$ such that a set $\{K_{1}, \ldots, K_{m}\}$ forms a face of $\Gamma_{r}$ if and only if 
$$
\# ( K_{1} \cup  \cdots \cup K_{m})  \leq g(r).
$$
Note that this implies that the size of the smallest $S \notin \Gamma_{r}$ has $\lfloor \log_{2}(r) \rfloor + 2$ elements (obtained by taking that many disjoint triplets).

The fact that $S_{1}|_{K} = S_{2}|_{K}$ for all $K \in { [n] \choose k}$ with $k \leq g(r)$ implies that 
$\pi_{\Gamma_{r}}(u_{S_{1}}) = \pi_{\Gamma_{r}}(u_{S_{2}})$.  Indeed, if $L = \{K_{1}, \ldots, K_{m}\}$, then $\pi_{L}(e_{T})$ is a table with a single nonzero entry which records which of the rooted triplets on $K_{1}, \ldots, K_{m}$, that tree $T$ has.  Thus, $\pi_{L}(u_{S_{i}})$ is a table which records which combinations of rooted triplets on $K_{1}, \ldots, K_{m}$ appear in the trees in $S_{i}$.  If $\# (K_{1} \cup \cdots \cup K_{m}) \leq g(r)$, this information can be read off from
$S_{1}|_{K_{1} \cup \cdots \cup K_{m}} = S_{2}|_{K_{1} \cup \cdots \cup K_{m}}$, and must be the same for both $S_{1}$ and $S_{2}$.

Since $\pi_{\Gamma_{r}}(u_{S_{1}}) = \pi_{\Gamma_{r}}(u_{S_{2}})$, we see that $u_{S_{1}} - u_{S_{2}} \in \ker_{\zz} \pi_{\Gamma_{r}}$.  Since $S_{1} \neq S_{2}$, $v = u_{S_{1}} - u_{S_{2}} \neq 0$.  By Theorem \ref{thm:kahle}, we have $\|v \|_{1} \geq 2^{\lfloor \log_{2}(r) \rfloor + 2} > 2r$.  On the other hand each $u_{S_{i}}$ has one norm $r$, so $\|v \|_{1} \leq 2r$.  This is a contradiction.
\end{proof}

\begin{rmk}
Note that the same argument for the upper bound would work even if our trees were not binary, by using vector spaces of dimension $4^{[n] \choose 3}$, since arbitrary rooted trees are determined by their rooted triplets (of which there are four possibilities).
\end{rmk}

The lower bound $RD(r)  \geq 3 ( \lfloor \log_{2}(r) \rfloor + 1),$ can be deduced from an elegant construction of Humphries (\cite{Humphries2008} stated for unrooted trees), which we repeat here for completeness.

\begin{proof}[Proof that $RD(r)  \geq 3 ( \lfloor \log_{2}(r) \rfloor + 1)$]
Suppose first that $r = 2^{k-1}$.  Let $T$ be a fixed, but otherwise arbitrary rooted leaf labeled binary tree with
$k$ leaves, labeled by $[k]$.  We will construct sets of trees on $3k$ leaves which
prove the lower bound.  Now let $X$ be the leaf label set  $X =  \{a_{1}, b_{1}, c_{1}$, $\ldots,$
$a_{k}, b_{k}, c_{k}\}$.  On each triple we use two rooted trees $a_{i}| b_{i}c_{i}$
and $b_{i} | a_{i}c_{i}$, denoted $t^{i}_{0}$ and $t^{i}_{1}$ respectively.  Given a
list of trees $t_{\epsilon} = (t^{1}_{\epsilon_{1}}, \ldots, t^{k}_{\epsilon_{k}})$ with $\epsilon \in \{0,1\}^{k}$,  form the tree with $3k$ leaves $T_{\epsilon}$  by identifying the root on tree $t^{i}_{\epsilon_{i}}$ with the label $i$ on the leaf of $T$.

Now let 
\begin{eqnarray*}
S_{k,odd}  & = & \{ T_{\epsilon} : \epsilon \in \{0,1\}^{k}, \quad \sum_{i} \epsilon_{i} \equiv 1 \mod 2 \}   \\
S_{k,even}  & = & \{ T_{\epsilon} : \epsilon \in \{0,1\}^{k}, \quad \sum_{i} \epsilon_{i} \equiv 0 \mod 2 \}. 
\end{eqnarray*}  
For any subset $K \subseteq X$, with $\#K = 3k-1$, we have $S_{k,odd} |_{K}  = 
S_{k,even} |_{K}$.  Indeed, this $K$ omits one vertex, say $a_{k}$ so that both triples $a_{k}|b_{k}c_{k}$ and $b_{k} |a_{k}c_{k}$ collapse to an identical cherry on $b_{k}$ and $c_{k}$ in all trees.  Thus, the trees in both $S_{k,odd} |_{K}$ and 
$S_{k,even} |_{K}$ are determined by all vectors $\epsilon \in \{0,1\}^{k-1}$.
Note that $\# S_{k,odd} = \#S_{k,even} = 2^{k-1}$.  This implies that $RD(r) \geq 3k = 3 (\log_{2}(r) + 1).$  

If $ 2^{k-1} \leq r < 2^{k}$, we let $T'$ be an arbitrary tree with $3k$ leaves on the label set $X$, and let $S_{1}$ and $S_{2}$ be the multisets  $S_{1} = S_{k,odd} \cup \{T', \ldots, T' \}$ and  $S_{2} = S_{k,even} \cup \{T', \ldots, T' \}$, where we union with $r - 2^{k-1}$ copies of $T'$.  Then by the preceding argument $S_{1}|_{K} = S_{2}|_{K}$ for all subsets $K \subseteq X$ with $\#K = 3k-1$.
\end{proof}

\section*{Acknowledgement}
\label{sec:acknowledgement}
  Seth Sullivant was partially supported by the David and Lucille Packard Foundation and the US National Science Foundation (DMS 0954865).

%%%%%%%%%%%%%%%%%%%%%%%%%%%%%%%%%%%%%%
%%%%%%%%%%%%%%%%%%%%%%%%%%%%%%%%%%%%%%
%%%%%%%%%%%%%%%%%%%%%%%%%%%%%%%%%%%%%%
%%%%%%%%%%%%%%%%%%%%%%%%%%%%%%%%%%%%%%

\end{document}